\documentclass[reqno]{amsart}
\usepackage[ansinew]{inputenc}
\usepackage[english]{babel}
\usepackage[all]{xy}
\usepackage{amsmath,latexsym,amssymb,verbatim}

\newtheorem{thm}{Theorem}[section]
\newtheorem{prop}[thm]{Proposition}

\newtheorem{cor}[thm]{Corollary}
\numberwithin{equation}{section}
\theoremstyle{definition}
\newtheorem{definition}[thm]{Definition}
\newtheorem{remark}[thm]{Remark}

\newcommand{\cB}{{\mathcal B}}
\newcommand{\cL}{{\mathcal L}}
\newcommand{\cE}{{\mathcal E}}

\newcommand{\cO}{{\mathcal O}}

\newcommand{\cM}{{\mathcal M}}
\newcommand{\cN}{{\mathcal N}}
\newcommand{\cK}{{\mathcal K}}

\newcommand{\bbQ}{{\mathbb Q}}
\newcommand{\bbZ}{{\mathbb Z}}

\newcommand{\bR}{{\mathbb R}}

\newcommand{\Ker}{\operatorname{Ker}}

\newcommand{\bbP}{\mathbb P}

\newcommand{\SHom}{{\mathcal{H}om\,}}
\newcommand{\Proj}{\operatorname{Proj}}

\newcommand{\punto}{{\displaystyle \cdot}}

\newcommand{\Kos}{\operatorname{Kos}}
\newcommand{\DeRham}{\operatorname{DeRham}}
\newcommand{\di}{\operatorname{d}}
\newcommand{\Spec}{\operatorname{Spec}}
\newcommand{\limi}{\operatorname{lim}}
\newcommand{\plim}[1]{\,\underset{#1}{\underset{\leftarrow}{\limi}}\,}
\newcommand{\wt}{\widetilde}

\newcommand{\marginnote}[1]{\ifthenelse{\isodd{\thepage}}{\normalmarginpar}
{\reversemarginpar}\marginpar{\fbox{\parbox{18mm}{\sloppy\footnotesize #1}}}}

\begin{document}
\title[Euler Sequence and Koszul complex of a module]{Euler Sequence and Koszul complex of a module}

\date{\today}
\author[B. Andreas]{Bj\"orn Andreas}
\author[D. S\'anchez G\'omez]{Dar\'io S\'anchez G\'omez}
\author[F. Sancho de Salas]{Fernando Sancho de Salas}
\address{Mathematisches Institut,
Freie Universit\"at Berlin, Arnimallee 4, Berlin  Germany}
\email{andreasb@math.fu-berlin.de}
\address{Departamento de Matem\'aticas and Instituto Universitario de F\'{\i}sica Fundamental y Matem\'aticas
(IUFFyM), Universidad de Salamanca, Plaza de la Merced 1-4, 37008
Salamanca, Spain.}
\email{dario@usal.es}\email{fsancho@usal.es}

\begin{abstract} We construct relative and global Euler sequences of a module. We apply it to prove some acyclicity results of the Koszul complex of a module and to compute the cohomology of the sheaves of (relative and absolute) differential $p$-forms of a projective bundle. In particular we generalize Bott's formula for the projective space to a projective bundle over a scheme of characteristic zero.
\end{abstract}

\thanks{This work was supported by the SFB 647 `Space-Time-Matter:Arithmetic and Geometric Structures' of the DFG (German Research Foundation) and by the Spanish grants MTM2013-45935-P (MINECO) and FS/12-2014 (Samuel Sol\'orzano Barruso Foundation).}

\vskip.5cm

\maketitle

\section*{Introduction}

This paper deals with two related questions: the acyclicity of the Koszul complex of a module and the cohomology of the sheaves of (relative and absolute) differential $p$-forms of a projective bundle over a scheme.

Let $M$ be a module over a commutative ring $A$. One has the Koszul complex $\Kos (M)=\Lambda^\punto M \otimes_A S^\punto M$, where $\Lambda^\punto M$ and $S^\punto M$ stand for the exterior and symmetric algebras of $M$. It is a graded complex $\Kos(M)=\bigoplus_{n\geq 0}\Kos(M)_n$, whose $n$-th graded component $\Kos(M)_n$ is the complex: \[ 0\xrightarrow{} \Lambda^nM\xrightarrow{}\Lambda^{n-1}M\otimes M \xrightarrow{} \Lambda^{n-2}M\otimes S^2 M \xrightarrow{}\cdots \xrightarrow{} M\otimes S^{n-1}M\xrightarrow{} S^nM \xrightarrow{}0\] It has been known for many years that $\Kos(M)_n$ is acyclic for $n>0$, provided that $M$ is a flat $A$-module or $n$ is invertible in $A$ (see \cite{Bourbaki} or \cite{Quillen67}). It was conjectured in \cite{SS00} that $\Kos (M)$ is always acyclic. A counterexample in characteristic 2 was given in \cite{GPV07}, but it is also proved there that $H_\mu(\Kos(M)_\mu)=0$ for any $M$, where $\mu$ is the minimal number of generators of $M$. Leaving aside the case of characteristic 2 (whose pathology is clear for the exterior algebra), we prove two new evidences for the validity of the conjecture (for $A$  noetherian): firstly, we prove (Theorem \ref{thm:Koszulacyclic-n>>0}) that, for any finitely generated $M$, $\Kos(M)_n$ is acyclic for $n>>0$; secondly, we prove (Theorem \ref{thm:KoszulacyclicIdeal}) that if $I$ is an ideal locally generated by a regular sequence, then  $\Kos(I)_n$ is acyclic for any $n>0$. These two results are a consequence of relating the Koszul complex $\Kos(M)$ with the geometry of the space $\bbP=\Proj S^\punto M$, as follows:

First of all, we shall reformulate the Koszul complex in terms of differential forms of $S^\punto M$ over $A$: the canonical isomorphism $\Omega_{S^\punto M/A}=M\otimes_A S^\punto M$ allows us to interpret the Koszul complex $\Kos(M)$ as the complex of differential forms $\Omega^\punto_{S^\punto M/A}$ whose differential, $i_D\colon \Omega^p_{S^\punto M/A}\to \Omega^{p-1}_{S^\punto M/A}$,  is the inner product with the $A$-derivation $D\colon S^\punto M\to S^\punto M$ consisting in multiplication by $n$ on $S^nM$. By homogeneous  localization, one obtains a complex of $\cO_\bbP$-modules $\wt\Kos(M)$ on $\bbP$. Our first result (Theorem \ref{thm:Kostildeacyclic}) is that the complex $\wt\Kos(M)$ is acyclic with factors (cycles or boundaries)  the sheaves $\Omega^p_{\bbP/A}$.  Moreover,  one has a natural morphism
\[ \Kos(M)_n\to \pi_*[\wt\Kos(M)\otimes \cO_{\bbP} (n)] \] with $\pi\colon \bbP\to \Spec A$ the canonical morphism. In Theorem \ref{thm:n-Koszulacyclic} we give (cohomological) sufficient conditions for the acyclicity of the complexes $\Kos(M)_n$ and $\pi_*[\wt\Kos(M)\otimes \cO_{\bbP} (n)]$. These conditions, under noetherian hypothesis, are satisfied for $n>>0$, thus obtaining Theorem \ref{thm:Koszulacyclic-n>>0}. The acyclicity of the Koszul complex of a locally regular ideal follows then from Theorem \ref{thm:n-Koszulacyclic} and the theorem of formal functions.

The advantage of expressing the Koszul complex $\Kos(M)$ as $(\Omega^\punto_{S^\punto M/A},i_D)$ is two-fold. Firstly, it makes clear its relationship with the De Rham complex $(\Omega^\punto_{S^\punto M/A},\di)$: The Koszul and De Rham differentials are related by Cartan's formula: $i_D\circ \di +\di\circ i_D=$ multiplication by $n$ on $\Kos(M)_n$. This yields a splitting result (Proposition \ref{prop:RelativeHomotopTrivial} or Corollary \ref{cor:relativesplitting}) which will be essential for some cohomological results in section \ref{sec:relativeBott} as we shall explain later on. Secondly, it allows a natural generalization (which is the subject of section \ref{sec:global}): If $A$ is a $k$-algebra, we define the complex $\Kos(M/k)$ as the complex of differential forms (over $k$), $\Omega^\punto_{S^\punto M/k}$ whose differential is the inner product with the same $D$ as before. Again, one has that $\Kos(M/k)=\bigoplus_{n\geq 0}\Kos(M/k)_n$ and it induces, by homogeneous localization, a complex $\wt\Kos(M/k)$ of modules on $\bbP$ which is also acyclic and whose factors are the sheaves $\Omega^p_{\bbP/k}$ (Theorem \ref{thm:globalKostildeacyclic}). We can reproduce the  aforementioned results about the complexes $\Kos(M)_n$, $\wt\Kos(M)$,   for the complexes $\Kos(M/k)_n$, $\wt\Kos(M/k)$.

Section \ref{sec:relativeBott} deals with the second subject of the paper: let $\cE$ be a locally free module of rank $r+1$ on a $k$-scheme $X$ and let $\pi\colon \bbP\to X$ be the associated projective bundle, i.e., $\bbP=\Proj S^\punto \cE$. There are well known results about the (global and relative) cohomology of the sheaves $\Omega^p_{\bbP/X}(n)$ and $\Omega^p_{\bbP/k}(n)$ (we are using the standard abbreviated notation $\cN(n)=\cN\otimes\cO_{\bbP}(n)$) due to Deligne, Verdier and Berthelot-Illusie (\cite{De73-SGA7}, \cite{Ver74},\cite{BI70}) and about the cohomology of the sheaves $\Omega^p_{\bbP_r}(n)$ of the ordinary projective space due to Bott (the so called Bott's formula, \cite{Bott57}). We shall not use their results; instead, we reprove them and we obtain some new results, overall when $X$ is a $\bbQ$-scheme. Let us be more precise:

In Theorem \ref{prop:VerdierRelative} we compute the relative cohomology sheaves $R^i\pi_*\Omega^p_{\bbP/X}(n)$, obtaining Deligne's result (see \cite{De73-SGA7} and also \cite{Ver74}) and  a new (splitting) result, in the case of a $\bbQ$-scheme, concerning the sheaves $\pi_*\Omega^p_{\bbP/X}(n)$ and $R^r\pi_*\Omega^p_{\bbP/X}(-n)$ for $n>0$. We obtain Bott formula for the projective space as a consequence. In Theorem \ref{prop:VerdierAbsolute} we compute the relative cohomology sheaves $R^i\pi_*\Omega^p_{\bbP/k}(n)$, obtaining Verdier's results (see \cite{Ver74}) and improving them in two ways: first, we give a more explicit description of $\pi_*\Omega^p_{\bbP/k}(n)$ and of $R^r\pi_*\Omega^p_{\bbP/k}(-n)$ for $n>0$; secondly, we obtain a splitting result for these sheaves when $X$ is a $\bbQ$-scheme (as in the relative case).

Regarding Bott's formula, we are able to generalize it for a projective bundle, computing the dimension of the cohomology vector spaces $H^q(\bbP,\Omega^p_{\bbP/X}(n))$ and $H^q(\bbP,\Omega^p_{\bbP/k}(n))$ when $X$ is a proper $k$-scheme of characteristic zero (Corollaries \ref{cor:relativebott} and \ref{cor:absolutebott}).

It should be mentioned that these results make use of the complexes $\wt\Kos(\cE)$ (as Deligne and Verdier) and $\wt\Kos(\cE/k)$. The complex $\wt\Kos(\cE)$ is essentially equivalent to the exact sequence
\[ 0\to \Omega_{\bbP/X}\to (\pi^*\cE)\otimes \cO_{\bbP}(-1) \to\cO_{\bbP}\to 0\] which is usually called Euler sequence. The complex $\wt\Kos(\cE/k)$ is equivalent to the exact sequence
\[ 0\to\Omega_{\bbP/k}\to\wt\Omega_{B/k}\to\cO_{\bbP}\to 0\] with $B=S^\punto\cE$, which we have called global Euler sequence.  These sequences still hold for any $A$-module $M$ (which we have called relative and global Euler sequences of $M$).
The aforementioned results about the acyclicity of the Koszul complexes of a module obtained in sections \ref{sec:relative} and \ref{sec:global} are a consequence of this fact.

\section{Relative Euler sequence of a module and Koszul complexes}\label{sec:relative}
Let $(X,\cO)$ be a scheme and let $\cM$ be   quasi-coherent $\cO$-module. Let  $\cB=S^\punto_{} \cM$ be the symmetric algebra of $\cM$ (over $\cO$), which is a graded $\cO$-algebra: the homogeneous component of degree $n$ is $\cB_n=S^n_{} \cM$. The module $\Omega_{\cB/\cO}$ of Kh\"aler differentials is a graded $\cB$-module in a natural way: $\cB\otimes_{\cO} \cB$ is a graded $\cO$-algebra, with $(\cB\otimes_{\cO}\cB)_n=\underset{p+q=n}\oplus\cB_p\otimes_{\cO}\cB_q$
and the natural morphism $\cB\otimes_{\cO}\cB\to\cB$ is a degree $0$ homogeneous morphism of graded algebras. Hence, the kernel $\Delta$ is a homogeneous ideal and $\Delta/\Delta^2=\Omega_{\cB/\cO}$ is a graded $\cB$-module. If $b_p,b_q\in \cB$ are homogeneous of degree $p,q$, then $b_p\di b_q$ is an element of $\Omega_{\cB/\cO}$ of degree $p+q$. We shall denote by $\Omega^p_{\cB/\cO}=\Lambda^p_{\cB} \Omega_{\cB/\cO}$, the $p$-th exterior power of $\Omega_{\cB/\cO}$, which is also a graded $\cB$-module in a natural way. For each $\cO$-module $\cN$, $\cN\otimes_{\cO} \cB$ is a graded $\cB$-module with gradation: $(\cN\otimes_{\cO} \cB)_n=\cN\otimes_{\cO} \cB_n$. Then one has the following basic result:

\begin{thm}\label{thm:relativediff}
The natural morphism of graded $\cB$-modules
$$\aligned \cM\otimes_\cO \cB[-1]&\to \Omega_{\cB/\cO}\\ m\otimes b&\mapsto b\di m\endaligned $$
is an isomorphism. Hence
$\Omega^p_{\cB/\cO}\simeq \Lambda^p\cM\otimes_{\cO}\cB[-p]$,
where  $\Lambda^i\cM=\Lambda^i_{\cO}\cM$.
\end{thm}

The natural morphism $\cM\otimes_{\cO}S^i\cM\to S^{i+1}_{}\cM$ defines a degree zero homogeneous morphism   of $\cB$-modules $\Omega_{\cB/\cO}=\cM\otimes_{\cO}\cB[-1]\to \cB$  which induces an $\cO$-derivation (of degree 0) $D\colon \cB\to \cB$, such that $\Omega_{\cB/\cO}\to \cB$ is the inner product with $D$. This derivation consists in multiplication by $n$ in degree $n$. It induces homogeneous morphisms of degree zero:
$$i_D\colon \Omega^{p}_{\cB/\cO}\to \Omega^{p-1}_{\cB/\cO}$$ and we obtain:

\begin{definition} The Koszul complex, denoted by $\Kos(\cM)$, is the complex:
\begin{equation}\label{eq:relativeKoszul}
\xymatrix@C=17pt{
\cdots \ar[r] & \Omega^{p}_{\cB/\cO}\ar[r]^(.5){i_D}&\Omega^{p-1}_{\cB/\cO}\ar[r]^(.6){i_D}&\cdots\ar[r]^(.3){i_D}&\Omega_{\cB/\cO}\ar[r]^(.55){i_D}& \cB\ar[r]  & 0
}
\end{equation}

Via Theorem \ref{thm:relativediff}, this complex is

$$\cdots \xrightarrow{i_D} \Lambda^{p}{\cM}\otimes_{\cO}\cB[-p]\xrightarrow{i_D}  \cdots \xrightarrow{i_D} \cM\otimes_{\cO}\cB[-1]\xrightarrow{i_D}   \cB\to 0$$

Taking the homogeneous components of degree $n\geq 0$, we obtain a complex of $\cO$-modules, which we denote by $\Kos(\cM)_n$:
\begin{equation*}
\xymatrix@C=17pt{
0\ar[r] & \Lambda^n\cM\ar[r] &  \Lambda^{n-1} \cM\otimes_{\cO} \cM\ar[r] &\dots\ar[r] & \cM\otimes_{\cO} S_{}^{n-1}\cM\ar[r] &S_{}^n\cM\ar[r] & 0
}
\end{equation*}
such that $\Kos(\cM)=\underset{n\geq 0}\oplus\Kos(\cM)_n$.
\end{definition}

Now let $\bbP=\Proj \cB$ and $\pi\colon\bbP \xrightarrow{} X$ the natural morphism. We shall use the following standard notations: for each $\cO_\bbP$-module $\cN$, we shall denote  $\cN(n)=\cN\otimes_{\cO_\bbP}\cO_\bbP(n)$ and for each graded $\cB$-module $N$ we shall denote by $\wt N$ the sheaf of $\cO_\bbP$-modules obtained by  homogeneous localization. We shall use without mention the following facts: homogeneous localization commutes with exterior powers and for any quasi-coherent module $\cL$ on $X$ one has $\wt{(\cL\otimes_\cO \cB[r])}=(\pi^*\cL)(r)$.

\begin{definition} Taking homogeneous localization on the Koszul complex (\ref{eq:relativeKoszul}), we obtain a complex of $\cO_{\bbP}$-modules, which we denote by $\widetilde\Kos(\cM)$:
\begin{equation}\label{eq:relativeKoszultilde}
\xymatrix@C=17pt{
\cdots \ar[r] & \widetilde\Omega^{d}_{\cB/\cO}\ar[r]^(.5){i_D}&\widetilde\Omega^{d-1}_{\cB/\cO}\ar[r]^(.6){i_D}&\cdots\ar[r]^(.34){i_D}&\widetilde\Omega_{\cB/\cO}\ar[r]^(.55){i_D}&\cO_\bbP\ar[r]  & 0
}
\end{equation}
\end{definition}

By Theorem \ref{thm:relativediff}, $\widetilde\Omega^{d}_{\cB/\cO}=(\pi^*\Lambda^d\cM)(-d)$, hence $\widetilde\Kos(\cM)$ can be written as
$$
\cdots \xrightarrow{i_D} (\pi^*\Lambda^d\cM) (-d) \xrightarrow{i_D} \cdots\to (\pi^*\cM) (-1)\xrightarrow{i_D} \cO_{\bbP} \to 0
$$

\begin{thm}\label{thm:Kostildeacyclic}
The complex $\widetilde\Kos(\cM)$ is acyclic (that is, an exact sequence). Moreover,
$$\Omega^p_{\bbP/X}=\Ker\Big( \wt\Omega^p_{\cB/\cO}\xrightarrow{i_D} \wt\Omega^{p-1}_{\cB/\cO}\Big).$$
Hence one has exact sequences
$$0\to\Omega^p_{\bbP/X}\xrightarrow{}\widetilde{\Omega}^p_{\cB/\cO}\xrightarrow{}\Omega_{\bbP/X}^{p-1}\to 0 $$
and right and left resolutions of $\Omega^p_{\bbP/X}$:
$$0\to\Omega^p_{{\bbP}/X}\xrightarrow{} \widetilde{\Omega}^p_{\cB/\cO}\xrightarrow{} \widetilde{\Omega}^{p-1}_{\cB/\cO}\xrightarrow{}\cdots\xrightarrow{} \widetilde{\Omega}_{\cB/\cO}\xrightarrow{}  \cO_{\bbP}\to 0$$
$$\cdots \to\widetilde{\Omega}^{r+1}_{\cB/\cO}\xrightarrow{} \widetilde{\Omega}^{r}_{\cB/\cO}\xrightarrow{}\cdots\xrightarrow{} \widetilde{\Omega}^{p+1}_{\cB/\cO}\xrightarrow{} \Omega^p_{{\bbP}/X}\xrightarrow{} 0$$
In particular, for $p=1$ the exact sequence
\begin{equation}\label{eq:RelativeEulerSeq}
0\to\Omega_{\bbP/X}\to \wt\Omega_{\cB/\cO}\to\cO_{\bbP}\to 0
\end{equation}
is called the (relative) Euler sequence.
\end{thm}

\begin{proof}
The morphism $\wt\Omega_{\cB/\cO}\to \cO_\bbP$ is surjective, since  $\cM\otimes_\cO  \cB[-1]\to \cB$ is surjective in positive degree. Let $K$ be the kernel. We obtain an exact sequence
\[ 0\to K\to \wt\Omega_{\cB/\cO}\to \cO_\bbP\to 0\] Since $\cO_\bbP$ is free, this sequence splits  locally; then, it induces exact sequences
\[ 0\to\Lambda^p K \to \wt\Omega^p_{\cB/\cO}\to \Lambda^{p-1}K\to 0\] Joining these exact sequences one obtains the Koszul complex $\widetilde\Kos (\cM)$. This proves the acyclicity of $\widetilde\Kos (\cM)$. To conclude, it suffices to prove that $K=\Omega_{\bbP/X}$.

Let us first define a morphism $\Omega_{\bbP/X}\to \widetilde\Omega_{\cB/\cO}$. Assume for simplicity that $X=\Spec A$. For each $b\in \cB$ of degree 1, let $U_b$ the standard affine open subset of $\bbP$ defined by $U_b=\Spec (\cB_{(b)})$, with $\cB_{(b)}$ the $0$-degree component of $\cB_b$. The natural inclusion $\cB_{(b)}\to \cB_b$ induces a morphism $\Omega_{\cB_{(b)}/A}\to \Omega_{\cB_b/A}=(\Omega_{\cB/A})_b$ which takes values in the $0$-degree component, $(\Omega_{\cB/A})_{(b)}$. Thus one has a morphism $\Omega_{\cB_{(b)}/A}\to (\Omega_{\cB/A})_{(b)}$, i.e. a morphism
$\Gamma(U_b,\Omega_{\bbP/X})\to \Gamma(U_b, \widetilde\Omega_{\cB/A}) $. One checks that these morphisms glue to a morphism $f\colon \Omega_{\bbP/X}\to \widetilde\Omega_{\cB/A}$. This morphism is injective, because the inclusion $\cB_{(b)}\to \cB_b$ has a retract, $c_n/b^k\mapsto c_n/b^n$, which induces a retract in the differentials. The composition $\Omega_{\bbP/X}\to \widetilde\Omega_{\cB/A}\to\cO_\bbP$ is null, as one checks in each $U_b$:
\[ (i_D\circ f)( \di(c_k/b^k))=i_D\left(\frac{b^k\di c_k- c_k\di b^k}{b^{2k}}\right)=\frac{b^ki_D\di c_k- c_k i_D\di b^k}{b^{2k}} =0\] because $i_D\di c_r=rc_r$ for any element $c_r$ of degree $r$. Thus, we have that $\Omega_{\bbP/X}$ is contained in the kernel of  $\widetilde\Omega_{\cB/A}\to \cO_{\bbP}$. To conclude, it is enough to see that the image of $\widetilde\Omega^2_{\cB/A}\overset{i_D}\to \widetilde\Omega_{\cB/A}$ is contained in $\Omega_{\bbP/X}$. Again, this is a computation in each $U_b$; one checks the equality
\[ i_D\left(\frac {\di c_p\wedge \di c_q}{b^{p+q}}\right)= p\frac{c_p}{b^p}\di \left(\frac{c_q}{b^q}\right)-q \frac{c_q}{b^q}\di \left(\frac{c_p}{b^p}\right)\] and the right member belongs to $\Omega_{\cB_{(b)}/A}$.
\end{proof}

For each $n\in \bbZ$, we shall denote by $\widetilde\Kos(\cM) (n)$ the complex $\widetilde\Kos(\cM)$ twisted by $\cO_\bbP(n)$ (notice that the differential of the Koszul complex is $\cO_\bbP$-linear). The differential of the complex  $\widetilde\Kos(\cM) (n)$ is still denoted by $i_D$.

\subsection{Acyclicity of the Koszul complex of a module}$\,$
\medskip

Let us  denote $\widetilde\Kos(\cM)_n :=\pi_*(\widetilde\Kos(\cM) (n))$. The natural morphisms $[\Omega^p_{\cB/\cO}]_n \to \pi_*[\wt\Omega^p_{\cB/\cO}(n)]$ give a morphism of complexes
\[ \Kos(\cM)_n \to \widetilde\Kos(\cM)_n \] and one has:

\begin{thm}\label{thm:n-Koszulacyclic} Let  $\cM$ be a finitely generated quasi-coherent module on a scheme $(X,\cO)$, $\bbP=\Proj S^\punto \cM$ and $\pi\colon \bbP\to X$ the natural morphism. Let $d$ be the minimal number of generators of $\cM$ (i.e., it is the greatest integer such that $\Lambda^d \cM\neq 0$) and $n> 0$. Then:
\begin{enumerate}
\item If $ R^j \pi_*[ \wt\Omega^i_{\cB/\cO}(n)] =0$ for any $j>0$ and any $0\leq i\leq d$, then $\widetilde\Kos(\cM)_n $ is acyclic.

 \item If (1) holds and the natural morphism $[\Omega^i_{\cB/\cO}]_n \to \pi_*[\wt\Omega^i_{\cB/\cO}(n)]$ is an isomorphism for any $0\leq i\leq d$, then $\Kos(\cM)_n $ is also acyclic. \end{enumerate}
\end{thm}

\begin{proof}
$(1)$ By Theorem \ref{thm:Kostildeacyclic}, the complex $\widetilde\Kos(\cM)  (n)$  is acyclic. Since the (non-zero) terms of this complex are $ \wt\Omega^i_{\cB/\cO}(n)$, the hypothesis tells us that $\pi_*(\widetilde\Kos(\cM)\otimes \cO_\bbP(n))$ is acyclic, that is, $\widetilde\Kos(\cM)_n $ is acyclic.

$(2)$ By hypothesis,  $\Kos(\cM)_n \to \widetilde\Kos(\cM)_n $ is an isomorphism and then $\Kos(\cM)_n $ is also acyclic.
\end{proof}

\begin{thm}\label{thm:Koszulacyclic-n>>0} Let $X$ be a noetherian scheme and $\cM$ a coherent module on $X$. The Koszul complexes $\Kos(\cM)_n$ and $\widetilde\Kos(\cM)_n  $ are acyclic for $n>>0$.
\end{thm}

\begin{proof} Indeed,  the hypothesis $(1)$ and $(2)$ of Theorem \ref{thm:n-Koszulacyclic} hold for $n>>0$ (see \cite[Theorem 2.2.1]{EGAIII-1} and \cite[Section 3.3 and Section 3.4]{EGAII}).
\end{proof}

\begin{thm}\label{thm:KoszulacyclicIdeal} Let $I$ be an ideal of a noetherian ring $A$. If $I$ is locally generated by a regular sequence, then $\Kos(I)_n$ and $\widetilde\Kos(I)_n $ are acyclic for any $n>0$.
\end{thm}

\begin{proof} In this case $\pi\colon  \bbP\to X=\Spec A$ is the blow-up with respect to $I$, because $S^n I=I^n$, since $I$ is locally a regular ideal (\cite{Mic64}). Let $d$ be the minimum number of generators of $I$.
By Theorem \ref{thm:n-Koszulacyclic}, it suffices to see that for any $A$-module $M$ and any $0\leq i\leq d$ one has:
\[ H^j(\bbP, (\pi^*M) (n-i))=\left\{\aligned \quad 0\quad \quad,&\text{ if } j>0\\ M\otimes_AI^{n-i}, &\text{ if } j=0\endaligned\right.\]

This is a consequence of the Theorem of formal functions (see \cite[Corollary 4.1.7]{EGAIII-1}). Indeed,   let us denote $Y_r=\Spec A/I^r$,  $E_r=\pi^{-1}(Y_r)$ and $\pi_r\colon E_r\to Y_r$. One has that $E_r=\Proj S^\punto_{A/I^r}(I/I^{r+1})$ is a projective bundle  over $Y_r$, because $I/I^{r+1}$ is a locally free $A/I^r$-module of rank $d$, since $I$ is locally regular. Hence, for any module $N$ on $Y_r$ and any $m> -d$ one has
\[ H^j(E_r,(\pi_r^*N) (m))=\left\{\aligned \quad 0\quad\quad\quad ,&\text{ if } j>0\\ N\otimes_{A/I^r}I^{m}/I^{m+r}, &\text{ if } j=0\endaligned\right.\]

Now, by the theorem of formal functions (let $m=n-i$)
\[ H^j(\bbP, (\pi^*M) (m))^\wedge = \plim{r} H^j(E_r, \pi_r^*(M/I^rM) (m))=0, \text { for } j>0.\]

For $j=0$, the natural morphism $M\otimes_A I^m\to H^0(\bbP, (\pi^*M)  (m))$ is an isomorphism because it is an isomorphism after completion by $I$:

\[ \aligned H^0(\bbP, (\pi^*M) (m))^\wedge &= \plim{r} H^0(E_r, \pi_r^*(M/I^rM) (m))\\ &= \plim{r} (M/I^rM)\otimes_{A/I^r} I^m/I^{m+r}\\ &=\plim{r} (M\otimes_AS^mI)\otimes_A A/I^r = (M\otimes_AI^m)^\wedge \endaligned\]

\end{proof}

\begin{remark} Let $d$ be the minimum number of generators of $\cM$. Since $\widetilde\Kos(\cM)$ is acyclic and $\pi_*$ is left exact, one has that $H_d(\widetilde\Kos(\cM)_n)=0$ for any $n$. One the other hand, it is proved in \cite{GPV07} that $H_d(\Kos(\cM)_d)=0$. One cannot expect $\Kos(\cM)_n\to \widetilde\Kos(\cM)_n$ to be an isomorphism in general. For instance, consider $X=\Spec A$ with $A=k[u,v,s_1,s_2,t_1,t_2]/I$ where $k$ is a field and $I=(-us_1+vt_1+ut_2,vs_1+us_2-vt_2,vs_2,ut_1)$. Let $M= (Ax\oplus Ay)/A(\bar ux+\bar vy)$, where $\bar u$ (resp. $\bar v$) is the class of $u$ (resp. $v$) in A. Then one can prove that the map $M\to \pi_*\cO_{\bbP}(1)$ is not injective (for details we refer to section 26.21 of The Stacks project). So that the question which arises here is whether  $\Kos(\cM)_n\to \widetilde\Kos(\cM)_n$ is a quasi-isomorphism. We do not know the answer, besides  the acyclicity theorems for both complexes mentioned above.
\end{remark}

\subsection{Koszul versus De Rham}\label{sec:DeRham}

The exterior differential defines morphisms $$\di\colon \Omega^p_{\cB/\cO}\to \Omega^{p+1}_{\cB/\cO}$$ which are $\cO$-linear, but not $\cB$-linear. One has then the De Rham complex:

\[ \DeRham(\cM)\equiv 0\to \cB\overset\di \to \Omega_{\cB/\cO}\overset\di \to\cdots\overset\di\to\Omega^{p}_{\cB/\cO}\overset\di\to \Omega^{p+1}_{\cB/\cO}\to \cdots\] which can be reformulated as
\[ 0\to \cB\overset\di \to \cM\otimes_\cO \cB[-1] \to\cdots \Lambda^{p}\cM\otimes_\cO \cB[-p] \to \Lambda^{p+1} \cM\otimes_\cO \cB[-p-1]\to \cdots\]
Taking into account that $\di$ is homogeneous of degree 0, one has for each $n\geq 0$ a complex  of $\cO$-modules
\[ \DeRham(\cM)_n \equiv 0\to S^n\cM\to \cM\otimes_\cO S^{n-1}\to\cdots \to \Lambda^{n-1}\cM\otimes_\cO \cM\to \Lambda^n\cM\to 0\]

The differentials of the Koszul and De Rham complexes are related by Cartan's formula: $i_D\circ \di+ \di\circ i_D=$  multiplication by $n$ on $\Lambda^p\cM\otimes_\cO S^{n-p}\cM$. This immediately implies the following result:

\begin{prop} If $X$ is a scheme over $\bbQ$, then $\Kos(\cM)_n $ and $\DeRham(\cM)_n$ are homotopically trivial for any $n>0$. In particular, they are acyclic.
\end{prop}

Now we pass to homogeneous localizations. The differential $\di\colon \Omega^p_{\cB/\cO}\to \Omega^{p+1}_{\cB/\cO}$ is compatible with homogeneous localization, since for any $\omega_{k+n}\in \Omega^p_{\cB/\cO}$ of degree $k+n$ and any $b\in \cB$ of degree $1$, one has:
\[ \di\left( \frac {\omega_{k+n}}{b^n}\right)= \frac{b^n\di \omega_{k+n}- (\di b^n)\wedge \omega_{k+n}}{b^{2n}}.\]
Thus, for any $n\in \bbZ$, one has ($\cO$-linear) morphisms of sheaves $$\di\colon \widetilde\Omega^p_{\cB/\cO}  (n)\to \widetilde\Omega^{p+1}_{\cB/\cO} (n)$$ and we obtain, for each $n$, a complex of sheaves on $\bbP$:
\[ \widetilde\DeRham (\cM,n) =   0\to \cO_\bbP (n) \overset\di \to \widetilde\Omega_{\cB/\cO}  (n)\overset\di \to\cdots\overset\di\to \widetilde\Omega^{p}_{\cB/\cO} (n)\to \cdots\]
which can be reformulated as
$$0\to \cO_\bbP(n) \overset\di \to (\pi^*\cM) (n-1) \to\cdots   \to (\pi^*\Lambda^p\cM)  (n-p)\to \cdots$$
It should be noticed that $\widetilde\DeRham (\cM,n)$ is not the complex obtained for $n=0$ twisted by $\cO_\bbP(n)$, because the differential is not $\cO_\bbP$-linear.

Again, one has that $i_D\circ\di+\di\circ i_D=$  multiplication by $n$, on $\wt\Omega^p_{\cB/\cO} (n)$. Hence, one has:

\begin{prop}\label{prop:RelativeHomotopTrivial} If $X$ is a scheme over $\bbQ$, then the complexes $\widetilde\Kos (\cM) (n)$ and $\widetilde\DeRham (\cM,n)$ are homotopically trivial for any $n\neq 0$.
\end{prop}

\begin{cor}\label{cor:relativesplitting} Let $X$ be a scheme over $\bbQ$. For any $n\neq 0$, the exact sequences $$0\to \Omega^p_{\bbP/X} (n)\to \wt\Omega^p_{\cB/\cO}(n)\to \Omega^{p-1}_{\bbP/X} (n)\to 0$$ split as sheaves of $\cO$-modules (but not as $\cO_\bbP$-modules).
\end{cor}

\medskip
\section{Global Euler sequence of a module and Koszul complexes}\label{sec:global}

Assume that $(X,\cO)$ is a $k$-scheme, where $k$ is a ring (just for simplicity, one could assume that $k$ is another scheme). Let $\cM$ be an $\cO$-module and  $\cB=S^\punto \cM$ the symmetric algebra over $\cO$. Instead of considering the module of K\"ahler differentials of $\cB$ over $\cO$, we shall now consider the module of K\"ahler differentials over $k$, that is, $\Omega_{\cB/k}$. As it happened with $\Omega_{\cB/\cO}$ (section \ref{sec:relative}), the module $\Omega_{\cB/k}$ is a graded $\cB$-module in a natural way. The $\cO$-derivation $D\colon \cB\to \cB$ is in particular a $k$-derivation, hence it defines a morphism $i_D\colon \Omega_{\cB/k}\to \cB$, which is nothing  but the composition of the natural morphism $\Omega_{\cB/k}\to\Omega_{\cB/\cO}$ with the inner product $i_D\colon \Omega_{\cB/\cO}\to \cB$ defined in section \ref{sec:relative}.  Again we obtain a complex of $\cB$-modules $(\Omega^\punto_{\cB/k},i_D)$ which we denote by $\Kos(\cM/k)$:

\begin{equation}\label{eq:absoluteKoszul}
\xymatrix@C=17pt{
\cdots\ar[r] & \Omega^{p}_{\cB/k}\ar[r]^(.5){i_D}&\Omega^{p-1}_{\cB/k}\ar[r]^(.6){i_D}&\cdots\ar[r]^(.3){i_D}&\Omega_{\cB/k}\ar[r]^(.55){i_D}& \cB\ar[r]  & 0
}
\end{equation}
and for each $n\geq 0$ a complex of $\cO$-modules
$$
\xymatrix{
\Kos(\cM/k)_n= \cdots \ar[r] & [\Omega^p_{\cB/k}]_n\ar[r]^{i_D} & \cdots\ar[r] & [\Omega_{\cB/k }]_n\ar[r]^{i_D} & S^n\cM \ar[r] & 0
}
$$

By homogeneous localization one has a complex of $\cO_\bbP$-modules, denoted by $\widetilde\Kos(\cM/k)$:

$$
\xymatrix{\cdots \ar[r] & \widetilde\Omega^p_{\cB/k }\ar[r]^{i_D}&\widetilde\Omega^{p-1}_{\cB/k}\ar[r]^{i_D}& \cdots\ar[r]& \widetilde\Omega_{\cB/k}\ar[r]^{i_D}& \cO_\bbP \ar[r] & 0
}
$$

\begin{thm}\label{thm:globalKostildeacyclic} The complex $\widetilde\Kos(\cM/k)$ is acyclic (that is, an exact sequence). Moreover,
$$\Omega^p_{\bbP/k}=\Ker\Big(  \widetilde\Omega^p_{\cB/k }\overset{i_D}\to \widetilde\Omega^{p-1}_{\cB/k }\Big).$$
Hence one has exact sequences
$$0\to\Omega^p_{\bbP/k}\xrightarrow{}\widetilde{\Omega}^p_{\cB/k}\xrightarrow{}\Omega_{\bbP/k}^{p-1}\to 0 $$
and right and left resolutions of $\Omega^p_{\bbP/k}$:
$$0\to\Omega^p_{{\bbP}/k}\xrightarrow{} \widetilde{\Omega}^p_{\cB/k}\xrightarrow{} \widetilde{\Omega}^{p-1}_{\cB/k}\xrightarrow{}\cdots\xrightarrow{} \widetilde{\Omega}_{\cB/k}\xrightarrow{}  \cO_{\bbP}\to 0$$
$$\cdots\to\widetilde{\Omega}^{e}_{\cB/k}\xrightarrow{} \widetilde{\Omega}^{e-1}_{\cB/k}\xrightarrow{}\cdots\xrightarrow{} \widetilde{\Omega}^{p+1}_{\cB/k}\xrightarrow{} \Omega^p_{{\bbP}/k}\xrightarrow{} 0$$

In particular, for $p=1$ the exact sequence
\begin{equation}\label{eq:GlobalEulerSeq}
0\to\Omega_{\bbP/k}\to \widetilde{\Omega}_{\cB/k}\to\cO_{\bbP}\to 0
\end{equation}
is called the (global) Euler sequence.
\end{thm}

\begin{proof} It is completely analogous to the proof of Theorem \ref{thm:Kostildeacyclic}.
\end{proof}

Let us  denote $\widetilde\Kos(\cM/k)_n :=\pi_*(\widetilde\Kos(\cM/k)(n))$. The natural morphisms
$$[\Omega^{p}_{\cB/k}]_n\to \pi_*\big(\widetilde{\Omega}^{p}_{\cB/k} (n)\big)$$
give a morphism of complexes
$$ \Kos(\cM/k)_n \to \widetilde\Kos(\cM/k)_n .$$

In complete analogy to the relative setting we have the following:
\begin{thm}\label{thm:n-globalKoszulacyclic} Let  $\cM$ be a finitely generated quasi-coherent module on a scheme $(X,\cO)$, $\cB= S^\punto \cM$, $\bbP=\Proj\cB $ and $\pi\colon \bbP\to X$ the natural morphism. Let $d'$ be the greatest integer such that $\Omega^{d'}_{\cB/k}\neq 0$ and $n> 0$. Then:
\begin{enumerate}
\item If $ R^j \pi_* (\widetilde{\Omega}^{i}_{\cB/k} (n)) =0$ for any $j>0$ and any $0\leq i\leq d'$, then $\widetilde\Kos(\cM/k)_n $ is acyclic.

 \item If (1) holds and the natural morphism $[\Omega^{i}_{\cB/k}]_n\to \pi_*\big(\widetilde{\Omega}^{i}_{\cB/k} (n)\big)$ is an isomorphism for any $0\leq i\leq d'$, then $\Kos(\cM/k)_n $ is also acyclic.
 \end{enumerate}
\end{thm}

\begin{thm} Let $X$ be a noetherian scheme and $\cM$ a coherent module on $X$. The Koszul complexes $\Kos(\cM/k)_n$ and $\widetilde\Kos(\cM/k)_n  $ are acyclic for $n>>0$.
\end{thm}

\subsection{Koszul versus De Rham (Global case)}\label{sec:DeRhamGlobal}
Now we pass to the De Rham complex (over $k$). The $k$-linear differentials
$$\di\colon \Omega^p_{\cB/k}\to \Omega^{p+1}_{\cB/k}$$
give a (global) De Rham complex
$$\DeRham(\cM/k)\equiv 0\to \cB\overset\di \to \Omega_{\cB/k}\overset\di \to\cdots\overset\di\to\Omega^{p-1}_{\cB/k}\overset\di\to \Omega^{p}_{\cB/k}\to \cdots$$ which is bounded if $X$ is of finite type over $k$. Since $\di$ is homogeneous of degree $0$, one has for each $n\geq 0$ a complex of $\cO$-modules (with $k$-linear differential)
$$ \DeRham(\cM/k)_n\equiv 0\to S^n\cM\overset\di \to [\Omega_{\cB/k}]_n\overset\di \to\cdots\overset\di\to  [\Omega^{p}_{\cB/k}]_n\to \cdots.$$
One has again Cartan's formula: $i_D\circ\di +\di\circ i_D=$ multiplication by $n$, on $[\Omega^p_{\cB/k}]_n$ and then:

\begin{prop} If $X$ is a scheme over $\bbQ$, then $\Kos(\cM/k)_n$ and $\DeRham(\cM/k)_n$ are homotopically trivial (in particular, acyclic) for any $n>0$.
\end{prop}

As in section \ref{sec:DeRham}, we can take homogeneous localizations: for each $n\in\bbZ$, the differentials $\Omega^p_{\cB/k}\to \Omega^{p+1}_{\cB/k}$ induce $k$-linear morphisms
$$\di\colon \widetilde\Omega^p_{\cB/k} (n) \to \widetilde\Omega^{p+1}_{\cB/k} (n)$$
and one obtains a complex of $\cO_\bbP$-modules (with $k$-linear differential)
$$\widetilde\DeRham(\cM/k,n) = 0\to \cO_\bbP(n)\overset\di \to \widetilde\Omega_{\cB/k} (n)\overset\di \to\cdots\overset\di\to \widetilde\Omega^{p}_{\cB/k} (n)\to \cdots$$

Again, the differentials of Koszul and De Rham complexes are related by Cartan's formula: $i_D\circ\di+ \di\circ i_D=$ multiplication by $n$, on $\widetilde\Omega^{p}_{\cB/k} (n)$, so one has:

\begin{prop} Let $X$ be a scheme over $\bbQ$. The complexes $\widetilde\Kos(\cM/k) (n)$ and $\widetilde\DeRham(\cM/k,n) $ are homotopically trivial (in particular, acyclic) for any $n\neq 0$.
\end{prop}

\begin{cor}\label{cor:globalsplitting}
If $X$ is a scheme over $\bbQ$, then for any $n\neq 0$, the exact sequences
$$0\to \Omega^p_{\bbP/k} (n)\to \widetilde\Omega^p_{\cB/k} (n)\to \Omega^{p-1}_{\bbP/k} (n)\to 0$$
 split as sheaves of $k$-modules (but not as $\cO_\bbP$-modules).
\end{cor}

\section{Cohomology of projective bundles}\label{sec:relativeBott}

In this section we assume that $\cE$ is a locally free sheaf of rank $r+1$ on a $k$-scheme $(X,\cO)$. Let $\cB=S^\punto\cE$ be its symmetric algebra over $\cO$ and $\bbP= \Proj \cB\xrightarrow{\pi} X$ the corresponding projective bundle. Our aim is to determine the cohomology of  the sheaves $\Omega^p_{\bbP/X}(n)$ and $\Omega^p_{\bbP/k}(n)$.

\subsection{Cohomology of $\Omega^p_{\bbP/X} (n)$} $\,$
\medskip

\noindent{\ \ \bf Notations:} In order to simplify some statements, we shall use the following conventions:
 \begin{enumerate}

 \item $S^p\cE=0$ whenever $p<0$, and analogously for exterior powers.

 \item For any integer $p$, we shall denote $\bar p= r+1-p$.

 \item For any $\cO$-module $\cM$, we shall denote by $\cM^*$ its dual: $\cM^*=\SHom(\cM,\cO)$.
 \end{enumerate}

We shall use the following well known result on the cohomology of a projective bundle:

\begin{prop}\label{basics} Let $n$ be a non negative integer. Then
\[ R^i\pi_*\cO_\bbP(n)=\left\{\begin{array}{ll} 0 &\text{ for }\ \ i\neq 0\\ S^n\cE &\text{ for } \ \ i=0
\end{array}\right.\] If $n$ is a positive integer, then
\[ R^i\pi_*\cO_\bbP(-n)=\left\{\begin{array}{ll} 0 &\text{ for } \ \ i\neq r\\ S^{n-r-1}\cE^*\otimes \Lambda^{r+1}\cE^* &\text{ for }\ \ i=r
\end{array}\right.\]
\end{prop}

We shall also use without further explanation a particular case of projection formula: for any quasi-coherent module $\cN$ on $X$ and any locally free module $\cL$ on $\bbP$ such that $R^j\pi_*\cL$ is locally free (for any $j$),  one has $$R^i\pi_*(\pi^*\cN \otimes \cL)=\cN\otimes R^i\pi_*\cL .$$

\begin{prop}\label{relativedifferentials} Let $n$ be a non negative integer. Then
\[ R^i\pi_*\wt\Omega^p_{\cB/\cO} (n)=\left\{\begin{array}{ll} 0  &\text{ for }\ \ i\neq 0\\ \Lambda^p\cE\otimes S^{n-p}\cE &\text{ for }\ \ i=0
\end{array}\right.  \] For any positive integer $n$, one has
\[ R^i\pi_*\wt\Omega^p_{\cB/\cO} (-n)=\left\{\begin{array}{ll} 0\qquad &\text{ for }\ \ i\neq r\\  \Lambda^{\bar p}\cE^*\otimes S^{n -\bar p}\cE^* & \text{ for } \ \ i=r \ \ \text{ with }\ \ \bar p=r+1-p
\end{array}\right.\]
\end{prop}

\begin{proof}  Since $\wt\Omega^p_{\cB/\cO}=(\pi^*\Lambda^p\cE)(-p)$, the results follows from Proposition \ref{basics}. For the second formula we have also used the natural isomorphism $\Lambda^{\bar p}\cE = \Lambda^p\cE^*\otimes\Lambda^{r+1}\cE$.
\end{proof}

\begin{remark}\label{rem:KoszulLocallyFree} Notice that $\Lambda^p\cE\otimes S^{n-p}\cE=[\Omega^p_{\cB/\cO}]_n$. Thus, Proposition \ref{relativedifferentials} and Theorem \ref{thm:n-Koszulacyclic} tell us that  $\Kos(\cE)_n\to \wt\Kos(\cE)_n$ is an isomorphism for any $n\geq 0$ and the Koszul complexes $\wt\Kos(\cE)_n$ and $\Kos(\cE)_n$ are acyclic for any $n>0$ (thus we obtain the well known fact of the acyclicity of the Koszul complex of a locally free module).
\end{remark}

Let us denote   by $\cK_{p,n}$ the kernels of the morphisms $i_D$ in  $\Kos(\cE)_n$, that is,  $$\cK_{p,n}:=\Ker\big(\Lambda^p \cE\otimes S_{}^{n-p}\cE\xrightarrow{}\Lambda^{p-1} \cE\otimes S^{n-p+1}_{}\cE\big)$$

One has the following result (see \cite{Ver74} or \cite[Expos\'e XI]{De73-SGA7} for different approaches).
\begin{thm}\label{prop:VerdierRelative}
Let $\cE$ be a locally free sheaf of rank $r+1$ on a $k$-scheme $(X,\cO)$ and $\bbP= \Proj S^\punto\cE\xrightarrow{\pi} X$ the corresponding projective bundle.

Let $n$ be a positive integer number.
\begin{enumerate}\itemsep.5cm

\item \[R^i \pi_*\Omega^p_{\bbP/X}=\left\{ \begin{array}{ll} \cO &\text{ if } \ \ 0\leq i=p\leq r\\  0&\text{ otherwise } \end{array} \right. .\]

\item \[ R^i\pi_*\Omega^p_{\bbP/X}(n) =\left\{ \begin{array}{ll} 0 & \text{ if }\ \  i\neq 0\\ \cK_{p,n} & \text{ if }\ \ i=0\end{array} \right. \]
and, if $X$ is a $\bbQ$-scheme, then $$\cK_{p,n}\oplus \cK_{p-1,n}=\Lambda^p\cE\otimes S^{n-p}\cE.$$

\item \[ R^i \pi_*\Omega^p_{\bbP/X}(-n)= \left\{ \begin{array}{ll} 0 & \text{ if }\ \ i\neq r\\ \cK_{r-p,n}^* & \text{ if } \ \ i=r\end{array} \right. \]
and, if $X$ is a $\bbQ$-scheme, then $$\cK_{r-p,n}^*\oplus \cK_{r-p+1,n}^*=\Lambda^{\bar p}\cE^*\otimes S^{n -\bar p}\cE^*.$$
\end{enumerate}
\end{thm}

\begin{proof} Let $n\geq 0$. By Theorem \ref{thm:Kostildeacyclic}
$$0\to\Omega^p_{\bbP/X}(n)\xrightarrow{} \widetilde{\Omega}^p_{\cB/\cO}(n)\xrightarrow{} \cdots\xrightarrow{} \widetilde{\Omega}_{\cB/\cO}(n)\xrightarrow{}  \cO_{\bbP}(n)\to 0$$ is a resolution of $\Omega^p_{\bbP/X}(n)$ by $\pi_*$-acyclic sheaves (by Proposition \ref{relativedifferentials}). One concludes then by Proposition \ref{relativedifferentials} and Remark \ref{rem:KoszulLocallyFree}.

(3) follows from (2) and (relative) Grothendieck duality: one has an isomorphism $\Omega_{\bbP/X}^p= \SHom(\Omega_{\bbP/X}^{r-p}, \Omega^r_{\bbP/X})$ and then $$\bR \pi_*\Omega_{\bbP/X}^p(-n)\simeq\bR \pi_*\SHom(\Omega_{\bbP/X}^{r-p}(n),\Omega^r_{\bbP/X})\simeq \bR\SHom(\bR \pi_*\Omega^{r-p}_{\bbP/X}(n)[r],\cO)$$
and one concludes by  (2).

Finally, the statements of (2) and (3) regarding the case that $X$ is a $\bbQ$-scheme follow from Corollary \ref{cor:relativesplitting}.
\end{proof}

\begin{cor} (Bott's formula) Let $\bbP_r$ be the projective space of dimension $r$ over a field $k$. Let $n$ be a positive integer number.
\begin{enumerate}\itemsep.5cm

\item \[\dim_kH^q(\bbP_r, \Omega^p_{\bbP_r})=\left\{ \begin{array}{ll} 1&\text{ if }\ \ 0\leq q=p\leq r \\ 0 & \text{ otherwise }  \end{array} \right.\]

\item \[ \dim_kH^q(\bbP_r, \Omega^p_{\bbP_r}(n)) =\left\{ \begin{array}{ll} 0 & \text{ if } \ \ q\neq 0\\ \binom{n+r-p}{n}\binom{n-1}{p}&\text{ if }\ \ q=0\end{array} \right.\]

\item \[  \dim_kH^q(\bbP_r, \Omega^p_{\bbP_r}(-n)) = \left\{ \begin{array}{ll} 0 & \text{ if } \ \ q\neq r\\ \binom{n+p}{n}\binom{n-1}{r-p} & \text{ if } \ \ q=r\end{array} \right.\]
\end{enumerate}
\end{cor}

\begin{proof} It follows from Theorem \ref{prop:VerdierRelative}, once one proves that $\dim_k \cK_{p,n}=\binom{n+r-p}{n}\binom{n-1}{p}$. From the exact sequence $0\to \cK_{p,n}\to \Lambda^p\cE\otimes S^{n-p}\cE\to \cK_{p-1,n}\to 0$ it follows that $\dim_k \cK_{p,n}+\dim_k \cK_{p-1,n} = \binom{r+1}{p}\binom{n-p+r}{r}$; hence it suffices to prove that
\[ \binom{n+r-p}{n}\binom{n-1}{p}+ \binom{n+r-p+1}{n}\binom{n-1}{p-1}= \binom{r+1}{p}\binom{n-p+r}{r}\] which is an easy computation if one writes $\binom ab=\frac{a!}{b!(a-b)!}$.
\end{proof}

\begin{remark} (1) We can give an interpretation of $H^0(\bbP_r, \Omega^p_{\bbP_r}(n))$ in terms of differentials forms of the polynomial ring $k[x_0,\dots,x_r]$; one has the exact sequence \[0\to H^0(\bbP_r, \Omega^p_{\bbP_r}(n))\to [\Omega^p_{k[x_0,\dots, x_r]/k}]_n\overset {i_D}\to [\Omega^{p-1}_{k[x_0,\dots, x_r]/k}]_n\] that is, $H^0(\bbP_r, \Omega^p_{\bbP_r}(n))$ are those $p$-forms $\omega_p\in \Omega^p_{k[x_0,\dots, x_r]/k}$ which are homogeneous of degree $n$ and such that $i_D\omega_p=0$, where $D=\sum_{i=0}^{r}x_i\frac{\partial}{\partial x_i}$.

(2) From the exact sequence
\[ 0\to H^0(\bbP_r, \Omega^p_{\bbP_r}(n))\to \Lambda^p\cE\otimes S^{n-p}\cE\to \cdots\to \cE\otimes S^{n-1}\cE\to S^n\cE\to 0\] we can give a different combinatorial expression of $\dim_k H^0(\bbP_r, \Omega^p_{\bbP_r}(n))$ (as Verdier does):
\[ \dim_k H^0(\bbP_r, \Omega^p_{\bbP_r}(n))= \sum_{i=0}^p (-1)^i\binom {r+1}{p-i}\binom {n+r-p+i}{r} .\]
\end{remark}

It follows from Theorem \ref{prop:VerdierRelative} that $H^q(\bbP,  \Omega^p_{\bbP/X})= H^{q-p}(X,\cO)$. For the twisted case we have the following:

\begin{cor}\label{cor:relativebott} Let $X$ be a proper scheme over a field $k$ of characteristic zero. Let $\cE$ be a locally free module on $X$ of rank $r+1$ and $\bbP=\Proj S^\punto\cE$ the associated projective bundle. Then, for any positive integer $n$, one has:

\begin{enumerate}

\item $ \dim_kH^q(\bbP , \Omega^p_{\bbP/X}(n)) =  \sum_{i=0}^p (-1)^{i}\dim H^q(X, \Lambda^{p-i}  \cE\otimes S^{n-p+i}_{}\cE) .$

\item $  \dim_kH^q(\bbP, \Omega^p_{\bbP/X}(-n)) =  \sum_{i=0}^{p} (-1)^{i}\dim H^{q-r}(X, \Lambda^{\bar p+i} \cE^*\otimes S^{n-\bar{p}-i}_{}\cE^*) \\ \text{ with }\bar p=r+1-p.$

\end{enumerate}
\end{cor}
\begin{proof}  (1) By Corollary \ref{cor:relativesplitting}, one has $$H^q(\bbP , \Omega^p_{\bbP/X}(n))\oplus H^q(\bbP , \Omega^{p-1}_{\bbP/X}(n)) = H^q(\bbP , \wt\Omega^p_{\cB/\cO}(n))$$
and $H^q(\bbP , \wt\Omega^p_{\cB/\cO}(n))=H^q(X, \Lambda^p\cE\otimes S^{n-p}\cE)$ by Proposition \ref{relativedifferentials}. Conclusion follows.

(2) is completely analogous.

\end{proof}

\subsection{Cohomology of $\Omega^p_{\bbP/k}(n)$}
$\,$

Let us consider the exact sequence of differentials
\[ 0\to \Omega_{X/k}\otimes_\cO\cB\to\Omega_{\cB/k}\to\Omega_{\cB/\cO}\to 0\]

This sequence  locally splits: indeed, if $\cE$ is trivial, then $\cE=E\otimes_k\cO$ and $\cB=B\otimes_k\cO$, with $B=S^\punto E$; hence,  $\Omega_{\cB/\cO}=\Omega_{B/k}\otimes_k\cO$ and there is a natural morphism $\Omega_{B/k}\otimes_k\cO\to \Omega_{\cB/k}$ which is a section of $\Omega_{\cB/k}\to\Omega_{\cB/\cO}$.

\begin{remark} The exact sequence is a sequence of graded $\cB$-modules, hence it gives an exact sequence of $\cO$-modules in each degree. In particular, in degree 0 one obtains an isomorphism $\Omega_{X/k}=[\Omega_{\cB/k}]_0$, and an exact sequence in degree 1:
\[ 0\to\Omega_{X/k}\otimes_\cO \cE\to \left[ \Omega_{\cB/k}\right]_1\to \cE\to 0\]
which is nothing but the Atiyah extension.
\end{remark}

Taking homogeneous localizations we obtain an exact sequence of $\cO_\bbP$-modules
\[ 0\to\pi^*\Omega_{X/k} \to\wt\Omega_{\cB/k}\to\wt\Omega_{\cB/\cO} \to 0\] which splits  locally (on $X$).

\begin{prop}\label{absolutedifferentials}  Let $n$ be a positive integer. Then:
\begin{enumerate}
\item
\[ R^i\pi_*\wt\Omega^p_{\cB/k} =\left\{\begin{array}{ll} 0  &\text{ for }\ \ i\neq 0,r\\ \Omega^p_{X/k}\,  &\text{ for }\ \ i=0 \\ \Omega^{p-r-1}_{X/k}  &\text{ for }\ \ i=r
\end{array}\right. .\]

\item
\[ R^i\pi_*\wt\Omega^p_{\cB/k} (n)=\left\{\begin{array}{ll} 0  &\text{ for } \ \ i\neq 0\\ \left[\Omega^p_{\cB/k}\right]_n & \text{ for } \ \ i=0
\end{array}\right.\]
\item $R^i \pi_* \widetilde{\Omega}^p_{\cB/k}(-n)=0$ for $i\neq r$ and $R^r \pi_* \widetilde{\Omega}^p_{\cB/k}(-n)$ is locally isomorphic to $\bigoplus_{q=0}^p ( \Omega_{X/k}^{p-q}\otimes\Lambda^{\bar q}\cE^*\otimes S^{n-\bar q}\cE^* )$, with $\bar q=r+1-q$.
\item Furthermore, if $X$ is a smooth $k$-scheme (of relative dimension $d$), then
$$R^r \pi_* \widetilde{\Omega}^p_{\cB/k}(-n)=\left[\Omega^{d+\bar p}_{\cB/k}\right]_n^*\otimes\Omega_{X/k}^d$$
\end{enumerate}
\end{prop}

\begin{proof} If $\cE$ is trivial, then  $\wt\Omega_{\cB/k}=\pi^*\Omega_{X/k}\oplus \wt\Omega_{\cB/\cO}$, so $\wt\Omega^p_{\cB/k}=\bigoplus_{q=0}^p \pi^*\Omega^{p-q}_{X/k}\otimes \wt\Omega^q_{\cB/\cO}$ and (1)-(3) follow from Proposition \ref{relativedifferentials} in this case. Since $\cE$ is locally trivial, we obtain the vanishing statements of (1)-(3).

(1) The natural morphism $\Omega^p_{X/k}\to \pi_*\wt\Omega^p_{\cB/k}$ is an isomorphism because it is locally so. The natural morphism $\wt \Omega^{r+1}_{\cB/k}\to \Omega^{r+1}_{\cB/\cO}$ gives a morphism $  R^r\pi_*\wt\Omega^{r+1}_{\cB/k}\to R^r\pi_*\wt\Omega^{r+1}_{\cB/\cO}=\cO$, which is an isomorphism because it is locally so. Finally, for any $p\geq 0$, the natural morphism $\wt\Omega^p_{\cB/k}\otimes \wt\Omega^{r+1}_{\cB/k}\to \wt\Omega^{p+r+1}_{\cB/k}$ induces a morphism $\pi_*(\wt\Omega^p_{\cB/k})\otimes R^r\pi_*(\wt\Omega^{r+1}_{\cB/k})\to R^r\pi_*\wt\Omega^{p+r+1}_{\cB/k}$, i.e. a morphism $\Omega^p_{X/k}\to R^r\pi_*\wt\Omega^{p+r+1}_{\cB/k}$, which is an isomorphism because it is locally so.

(2) The natural morphism $\left[\Omega^p_{\cB/k}\right]_n \to \pi_*\wt\Omega^p_{\cB/k} (n)$ is an isomorphism because it is locally so.

It only remains to prove (4), which is a consequence of (relative) Grothendieck duality. Indeed, notice that, under the smoothness hypothesis, $R^r\pi_*\widetilde{\Omega}^p_{\cB/k}(-n)$ is locally free, by (3). Hence, if suffices to compute its dual. This is given by duality: the relative dualizing sheaf is $\Omega^r_{\bbP/X}=\wt\Omega^{r+1}_{\cB/\cO}$ and one has isomorphisms $\wt \Omega^{d+r+1}_{\cB/k}=\wt\Omega^{r+1}_{\cB/\cO}\otimes \pi^*\Omega^d_{X/k}$ and $\SHom (\widetilde{\Omega}^p_{\cB/k}, \wt \Omega^{d+r+1}_{\cB/k}) = \wt\Omega^{d+\bar p}_{\cB/k}$; then:
\[\aligned \left[R^r\pi_*\widetilde{\Omega}^p_{\cB/k}(-n)\right] ^*&=\pi_* \SHom_\bbP(\widetilde{\Omega}^p_{\cB/k}(-n),\widetilde{\Omega}^{r+1}_{\cB/\cO})\\
&= \pi_*[ \SHom_\bbP(\widetilde{\Omega}^p_{\cB/k}(-n),\widetilde{\Omega}^{d+r+1}_{\cB/k})\otimes \pi^*(\Omega^d_{X/k})^*]\\ &=(\pi_*\wt\Omega^{d+\bar p}_{\cB/k}(n))\otimes (\Omega^d_{X/k})^*\overset{(2)}= [\Omega^{d+\bar p}_{\cB/k}]_n \otimes (\Omega^d_{X/k})^* .\endaligned\]

\end{proof}

\begin{cor} The Koszul complexes $\Kos(\cE/k)_n$ and $\widetilde\Kos(\cE/k)_n  $ are acyclic for $n> 0$ and $\Kos(\cE/k)_n\to \widetilde\Kos(\cE/k)_n  $ is an isomorphism for any $n\geq 0$.
\end{cor}

Let us denote by $\overline\cK_{p,n}$ the kernels of the morphisms $i_D$ in the Koszul complex $\Kos(\cE/k)_n$; that is,  $$\overline\cK_{p,n}:=\Ker\Big([\Omega^p_{\cB/k}]_n\xrightarrow{}[\Omega^{p-1}_{\cB/k}]_n\Big)$$

\begin{thm}\label{prop:VerdierAbsolute}
Let $\cE$ be a locally free sheaf of rank $r+1$ on a $k$-scheme $(X,\cO)$ and $\bbP= \Proj S^\punto\cE\xrightarrow{\pi} X$ the corresponding projective bundle.

Let $n$ be a positive integer. One has:
\begin{enumerate}
\item $ R^i\pi_* \Omega^p_{\bbP/k} = \Omega^{p-i}_{X/k}$.

\item
\[ R^i\pi_* \Omega^p_{\bbP/k} (n)=\left\{\begin{array}{ll} 0 &\text{ for } \ \ i\neq 0\\ \overline\cK_{p,n} &\text{ for }\ \ i=0
\end{array}\right.\] and, if $X$ is a $\bbQ$-scheme, then one has an isomorphism (of $k$-modules, not of $\cO$-modules)
$$\overline\cK_{p,n}\oplus \overline\cK_{p-1,n} = [\Omega^p_{\cB/k}]_n   .$$
\item $R^i \pi_*  {\Omega}^p_{\bbP/k}(-n)=0$ for $i\neq r$ and $R^r \pi_*  {\Omega}^p_{\bbP/k}(-n)$ is locally isomorphic to $\bigoplus_{q=0}^p  \Omega_{X/k}^{p-q}\otimes  \cK_{r-q,n}^*$. Moreover, if $X$ is a $\bbQ$-scheme, then one has an isomorphism (of $k$-modules, not of $\cO$-modules)
    $$R^r \pi_*  {\Omega}^p_{\bbP/k}(-n)\oplus R^r \pi_*  {\Omega}^{p-1}_{\bbP/k}(-n) = R^r\pi_*\wt\Omega^p_{\cB/k}(-n)$$
\item If $X$ is a smooth $k$-scheme (of relative dimension $d$), then
$$R^r \pi_*  {\Omega}^p_{\bbP/k}(-n)=\overline\cK_{d+r-p,n}^*\otimes\Omega_{X/k}^d$$ and, if $X$ is a $\bbQ$-scheme, then one has an isomorphism (of $k$-modules, not of $\cO$-modules)
$$R^r \pi_*  {\Omega}^p_{\bbP/k}(-n)\oplus R^r \pi_*  {\Omega}^{p-1}_{\bbP/k}(-n) = \left[\Omega^{d+\bar p}_{\cB/k}\right]_n^*\otimes\Omega_{X/k}^d. $$
\end{enumerate}
\end{thm}

\begin{proof}  If $\cE$ is trivial, then  $\Omega_{\bbP/k}=\pi^*\Omega_{X/k}\oplus  \Omega_{\bbP/X}$, so $ \Omega^p_{\bbP/k}=\bigoplus_{q=0}^p \pi^*\Omega^q_{X/k}\otimes  \Omega^{p-q}_{\bbP/X}$ and (1)-(3) follow from Theorem \ref{prop:VerdierRelative} in this case. Since $\cE$ is locally trivial, we obtain the vanishing statements of (1)-(3).

(1) The exact sequences $0\to \Omega^p_{\bbP/k}\to\wt\Omega^p_{\cB/k}\to  \Omega^{p-1}_{\bbP/k}\to 0$ induce morphisms
\[ \pi_* \Omega^{p-i}_{\bbP/k}\to  R^1\pi_*\Omega^{p-i+1}_{\bbP/k}\to\cdots \to R^i\pi_* \Omega^p_{\bbP/k}\] whose composition with the natural morphism $\Omega^{p-i}_{X/k} \to \pi_* \Omega^{p-i}_{\bbP/k}$ gives a morphism $\Omega^{p-i}_{X/k} \to R^i\pi_* \Omega^p_{\bbP/k}$. This morphism is an isomorphism because it is locally so.

(2) The exact sequence $0\to \Omega^p_{\bbP/k}(n) \to \wt\Omega^p_{\cB/k}(n)\to \wt\Omega^{p-1}_{\cB/k}(n)$ induces, taking direct image, the isomorphism $\pi_*\Omega^p_{\bbP/k}(n)= \overline\cK_{p,n}$.

 (4) follows from (2) and (relative) Grothendieck duality. Indeed, notice that, under the smoothness hypothesis, $R^r\pi_* {\Omega}^p_{\bbP/k}(-n)$ is locally free, by (3). Hence, if suffices to compute its dual. This is given by duality: the relative dualizing sheaf is $\Omega^r_{\bbP/X}$ and one has isomorphisms $\Omega^{d+r}_{\bbP/k}=\Omega^r_{\bbP/X}\otimes\pi^*\Omega^d_{X/k}$ and   $\SHom (\Omega^p_{\bbP/k},   \Omega^{d+r}_{\bbP/k}) =  \Omega^{d+r-p}_{\bbP/k}$; then:
\[\aligned \left[R^r\pi_* {\Omega}^p_{\bbP/k}(-n)\right] ^*&=\pi_* \SHom_\bbP( {\Omega}^p_{\bbP/k}(-n), {\Omega}^r_{\bbP/k})\\
&= \pi_*[ \SHom_\bbP( {\Omega}^p_{\bbP/k}(-n), {\Omega}^{d+r}_{\bbP/k})\otimes \pi^*(\Omega^d_{X/k})^*]\\ &=(\pi_*\wt\Omega^{d+ r - p}_{\bbP/k}(n))\otimes (\Omega^d_{X/k})^*\overset{(2)}=\overline \cK_{d+r-p,n} \otimes (\Omega^d_{X/k})^* .\endaligned\]

Finally, the statements of (2)-(4) regarding  the case of a $\bbQ$-scheme follow from Corollary \ref{cor:globalsplitting}.
\end{proof}

\begin{remark} For $n=1$ a little more can be said (as Verdier does): The natural morphism $\Omega^p_{X/k}\otimes\cE\to \pi_*\Omega^p_{\bbP/k}(1)$ is an isomorphism. Indeed, the exact sequence
\[ 0\to \Omega_{X/k}\otimes\cB\to\Omega_{\cB/k}\to \Omega_{\cB/\cO}\to 0\] induces for each $p$ an exact sequence
\[ 0\to \Omega^p_{X/k}\otimes\cB\to\Omega^p_{\cB/k}\to \Omega^{p-1}_{\cB/k}\otimes \Omega_{\cB/\cO}\to \Omega^{p-2}_{\cB/k}\otimes S^2\Omega_{\cB/\cO}\to \cdots\] and taking degree 1, an exact sequence
\[ 0\to \Omega^p_{X/k}\otimes\cE\to [\Omega^p_{\cB/k}]_1\to \Omega^{p-1}_{X/k}\otimes \cE\to 0\] On the other hand, taking $\pi_*$ in the exact sequence
\[ 0\to \Omega^p_{\bbP/k}(1)\to \wt\Omega^p_{\cB/k}(1)\to \Omega^{p-1}_{\bbP/k}(1)\to 0\] gives the exact sequence
\[ 0\to \pi_*\Omega^p_{\bbP/k}(1)\to [\Omega^p_{\cB/k}]_1 \to \pi_*\Omega^{p-1}_{\bbP/k}(1)\to 0.\] Thus, the isomorphism $\Omega^p_{X/k}\otimes\cE\to \pi_*\Omega^p_{\bbP/k}(1)$ is proved by induction on $p$.
\end{remark}

\begin{remark} It is known  (see \cite{BI70} or \cite{Gros85}) that $\bR\pi_*  \Omega^{p}_{\bbP/k}$ is decomposable, i.e., one has an isomorphism in the derived category $\bR\pi_*\Omega^{p}_{\bbP/k}= \bigoplus_{i=0}^r \Omega^{p-i}_{X/k}[-i] $. Let us see that, for $p\in [0,r]$, this is a consequence of Theorem \ref{thm:globalKostildeacyclic} and Proposition \ref{absolutedifferentials}. Indeed, by Theorem \ref{thm:globalKostildeacyclic}, one has the exact sequence
\[ 0\to \Omega^{p}_{\bbP/k}\to \wt\Omega^{p}_{\cB/k}\to \wt\Omega^{p-1}_{\cB/k}\to \cdots \to \wt\Omega_{\cB/k}\to\cO_{\bbP}\to 0\] and, by Proposition \ref{absolutedifferentials}, $\wt\Omega^{p-i}_{\cB/k}$ are $\pi_*$-acyclic for any $i\geq 0$ and $\pi_*\wt\Omega^{p-i}_{\cB/k}=\Omega^{p-i}_{X/k}$. Then
\[\bR\pi_*  \Omega^{p}_{\bbP/k}\equiv\, 0\to \Omega^{p}_{X/k}\to  \Omega^{p-1}_{X/k}\to \cdots \to  \Omega_{X/k}\to\cO\to 0\]
and, since the differential $i_D\colon \Omega^j_{X/k}\to \Omega^{j-1}_{X/k}$ is null, we obtain the result.
\end{remark}

The decomposability of $\bR\pi_*  \Omega^{p}_{\bbP/k}$ implies an isomorphism  $$H^q(\bbP, \Omega^p_{\bbP/k}) =\bigoplus_{i=0}^{r}H^{q-i}(X,\Omega^{p-i}_{X/k}) $$ For the twisted case we have the following:

\begin{cor}\label{cor:absolutebott} Let $X$ be a proper scheme over a field $k$ of characteristic zero. Let $\cE$ be a locally free module on $X$ of rank $r+1$ and $\bbP=\Proj S^\punto\cE$ the associated projective bundle. Then, for any positive integer $n$, one has:
\begin{enumerate}
\item  $\dim_k  H^q(\bbP, \Omega^p_{\bbP/k}(n))=\sum_{i=0}^p(-1)^{i}\dim_k H^q(X, [\Omega^{p-i}_{\cB/k}]_n)$.
\item If $X$ is smooth over $k$ of dimension $d$, then $$\dim_k H^q(\bbP , \Omega^p_{\bbP /k}(-n))= \sum_{i=0}^{d+r-p}(-1)^{ i }\dim_k H^{d+r-q}(X, [\Omega^{d+r-p-i}_{\cB/k}]_n).$$
\end{enumerate}
\end{cor}
\begin{proof} (1) By Corollary \ref{cor:globalsplitting}, $$H^q(\bbP, \Omega^p_{\bbP/k}(n))\oplus H^q(\bbP, \Omega^{p-1}_{\bbP/k}(n))= H^q(\bbP, \wt\Omega^p_{\cB/k}(n))$$
and $H^q(\bbP, \wt\Omega^p_{\cB/k}(n))=H^q(X, [\Omega^p_{\cB/k}]_n)$ by Proposition \ref{absolutedifferentials}. Conclusion follows.

(2)   follows from (1) and duality.

\end{proof}


\end{document}